\documentclass{mc}   
\renewcommand\datereceived{}
\renewcommand\dateaccepted{??}

\begin{document}

\markboth{N. Malekzadeh, E. Abedi }{Ricci-generalized pseudosymmetric $(\kappa, \mu)$-manifolds}
\title{Ricci-generalized pseudosymmetric  $(\kappa, \mu)$-contact metric manifolds
\thanks{This work was supported by the ...}}



\author[AUTHOR1 and AUTHOR2 and AUTHOR3]{N. Malekzadeh\affil{1}\comma\corrauth and E. Abedi\affil{2} }
\address{\affilnum{1}, \affilnum{2}\ Department of Mathematics, Azarbaijan Shahid Madani University, Tabriz 53751 71379, I. R. Iran\\ 
   } 
\emails{{\tt n.malekzadeh@azaruniv.edu} (N. Malekzadeh), {\tt esabedi@azaruniv.ac.ir} (E. Abedi), {\tt }  } 
%
%
%

\begin{abstract}
In this paper we classify Ricci-generalized pseudosymmetric  $(\kappa, \mu)$-contact metric manifolds in the sense of Deszcz .  
\end{abstract}

\keywords{pseudosymmetric manifolds, Ricci-generalized pseudosymmetric manifolds, $(\kappa, \mu)$-contact metric manifolds }

\ams{53D10, 53C35}

\maketitle

\section{Introduction}  
Semisymmetric manifolds, as a direct generalization of locally symmetric manifolds, was first studied by Cartan \cite{42}. The classification of semisymmetric manifolds was described by Szab$\acute{o}$ \cite{13,27}. Chaki \cite{1} and Deszcz \cite{6} introduced two different concept of a pseudosymmetric manifold. In both sense various properties of pseudosymmetric manifolds have been studied \cite{1,2,23,30,44,45,3,4,5}. Recently Shaikh and Kundu proved  the equivalency of Deszcz and Chaki pseudosymmetry \cite{47}.

A Riemannian manifold $(M, g)$ is called semisymmetric if for all $X, Y\in \mathfrak{X}(M)$
\begin{equation*}
R(X,Y)~.~R=0,
\end{equation*} 
where $\mathfrak{X}(M)$ is Lie algebra of vector fields on $M$ \cite{13}. Deszcz \cite{6} generalized the concept of semisymmetry and introduced pseudosymmetric manifolds.  for a symmetric $(0, 2)$-tensor field $B$ on $M$ and $X, Y\in \mathfrak{X}(M)$, we define the endomorphism $X\wedge_B Y$ of $\mathfrak{X}(M)$ by
\begin{equation}\label{EQ1} 
(X \wedge_B Y ) Z = B( Y, Z) X - B( X, Z) Y~~~~~Z\in \mathfrak{X}(M) . 
\end{equation}

For a $(0, k)$-tensor field $T$, $k\geq 1$ and a $(0, 4)$-tensor $\mathcal{R}$, the $(0, k + 2)$ tensor fields $\mathcal{R} .T$ and $Q(B, T)$ are defined by \cite{7,6}
\begin{equation}\label{EQ2} 
\begin{array}{ll} 
(\mathcal{R}.T)(X_1,...,X_k; X, Y) & =(R(X,Y).T)(X_1,...,X_k) \\
 & =-T(R(X,Y)X_1, X_2, ...,X_k)\\
 &-...-T(X_1, ...,X_{k-1},R(X,Y)X_k),
\end{array}
\end{equation}     
and  
\begin{equation}\label{EQ3}
 \begin{array}{ll}
Q(B,T)(X_1, ...,X_k; X, Y)&=((X\wedge_B Y).T)(X_1, ..., X_k)\\
&=-T((X \wedge_B Y)X_1, X_2, ..., X_k)\\ 
&-... -T(X_1,...,X_{k-1},(X \wedge_B Y)X_k),
\end{array}
\end{equation}  
where $R(X,Y)=[\nabla_X, \nabla_Y]-\nabla_{[X, Y]}$ is the corresponding $(1, 3)$-tensor of $\mathcal{R}$.

A Riemannian manifold $M$ is said to be pseudosymmetric if the tensors $R.R$ and $Q(g, R)$ are linearly dependent at every point of $M$, i.e.,
\begin{equation}\label{EQ4}
R.R = L_R Q(g, R). 
\end{equation}   
This is equivalent to
\begin{equation}\label{EQ5}
(R(X, Y ).R)(U, V, W) = L_R [((X \wedge_g Y ).R)(U, V, W)]
\end{equation}
holding on the set $U_R = \lbrace x\in M: Q(g, R)\neq 0 ~at~ x \rbrace $, where $L_R$ is a smooth function on $U_R$ \cite{6}.
The manifold $M$ is called pseudosymmetric of constant type if $L$ is constant. Particularly if $L_R=0$ then $M$ is a semisymmetric manifold. 

Let $S$ denotes the Ricci tensor of $M^{2n+1}$. The manifold $M$ is called Ricci-generalized pseudosymmetric \cite{44,45} if the tensors $R . R$ and $Q(S,R)$ are linearly dependent. This is equivalent to
\begin{equation}\label{EQ7}  
R .R =L Q(S, R),
\end{equation}  
holding on the set $U = \lbrace x\in M: Q(S,R)\neq 0~ at~ x\rbrace $, where $L$ is a smooth function on $U$ and \cite{6}: 
\begin{equation}\label{EQ8}
 \begin{array}{ll}
Q(S,R)(X_1, X_2, X_3; X, Y )=(X\wedge_ SY)R(X_1, X_2) X_3-R((X \wedge_S Y)X_1, X_2) X_3\\
-R(X_1, (X \wedge_S Y) X_2) X_3-R(X_1, X_2)(X \wedge_S Y) X_3.
\end{array}
\end{equation}
3-dimensional pseudosymmetric spaces of constant type have been studied by Kowalski and Sekizawa \cite{20,15,16,21}. Conformally flat pseudosymmetric spaces of constant type were classified by  Hashimoto and Sekizawa for dimension three \cite{14} and by Calvaruso for dimensions $>2$ \cite{22}. In dimension three, Cho and Inoguchi studied pseudosymmetric contact homogeneous manifolds \cite{2}. Cho et al. treated the conditions that 3-dimensional trans-Sasakians, non-Sasakian generalized $(\kappa, \mu)$-spaces and quasi-Sasakians manifolds be pseudosymmetric \cite{7}. Belkhelfa et al.  obtained some results on pseudosymmetric Sasakian space forms  \cite{7}. Finally some classes of pseudosymmetric contact metric 3-manifolds  
have been studied by Gouli-Andreou and Moutafi \cite{25,26}.

Papantoniou classified semisymmetric $(\kappa, \mu)$-contact metric manifolds \cite{12}. Pseudosymmetric $N(\kappa)$-contact metric manifold were studied by De and Jun \cite{30}. Thay also classified Ricci-generalized pseudosymmetric $N(\kappa)$-contact metric manifolds \cite{30}. As a generalization, in this paper, we study Ricci-generalized pseudosymmetric $(\kappa, \mu)$-contact metric manifolds.

This paper is organized as follows. After some preliminaries on $(\kappa, \mu)$-contact metric manifolds, in section 3 we characterize Ricci-generalized pseudosymmetric $(\kappa, \mu)$-contact metric manifolds and give an example.

\section{Preliminaries} 
A contact manifold is an odd-dimensional $C^{\infty}$ manifold  $M^{2n+1}$ equipped with a global 1-form $\eta$ such that $\eta\wedge (d\eta)^{n}\neq0$ everywhere. Since $d\eta$ is of rank $2n$, there exists a unique vector field $\xi$ on $M^{2n+1}$ satisfying $\eta(\xi)=1$ and $d\eta(\xi,X)=0$ for any $X\in \mathfrak{X}(M)$ is called the Reeb vector field or characteristic vector field of $\eta $. A Riemannian metric $g$ is said to be an associated metric if there exists a (1,1) tensor field $\varphi$ such that 
\begin{equation*}
d\eta(X,Y)=g(X,\varphi Y), \hspace{8mm} \eta(X)=g(X,\xi),  \hspace{8mm} \varphi ^{2}=-I+\eta \otimes\xi.
\end{equation*} 
The structure $(\varphi,\xi,\eta,g)$ is called a contact metric structure and a manifold $M^{2n+1}$ with a contact metric structure is said to be a contact metric manifold. Given a contact metric structure $(\varphi,\xi,\eta,g)$, we define a (1,1) tensor field $h$ by $h=(1/2) \mathcal{L}_{\xi}\varphi$  where $\mathcal{L}$ denotes the operator of Lie differentiation. On a contact metric manifold
$h$ is symmetric operator and 
\begin{equation}\label{EQ0165}
\nabla_X \xi=-\varphi X-\varphi hX.
\end{equation}
A contact metric manifold is said to be a Sasakian manifold if \begin{equation*}
(\nabla_X \varphi)Y=g(X, Y )\xi- \eta(Y )X.  
\end{equation*} 
In which case we have
 \begin{equation}\label{EQ10}
R(X, Y )\xi = \eta(Y )X - \eta(X)Y.
\end{equation}
The sectional curvature $K(\xi, X)$ of a plane section spanned by $\xi $ and a vector $X$ orthogonal to $\xi $ is called a $\xi$-sectional curvature, while the sectional curvature $K(X, \varphi X)$ is called a $\varphi$-sectional curvature.

The $(\kappa, \mu)$-nullity distribution of a contact metric manifold $M (\varphi,\xi,\eta,g)$ is a distribution \cite{11}
\begin{equation*}
\begin{array}{ll}
N(\kappa, \mu) : p \longrightarrow N_p(\kappa, \mu) =
\lbrace W \in T_pM | R(X, Y )W &= \kappa [g(Y,W)X - g(X,W)Y ]\\
&+\mu [g(Y,W)hX - g(X,W)hY ] \rbrace,
\end{array}
\end{equation*}
where $\kappa, \mu $ are real constants. Hence if the characteristic vector field $\xi $ belongs to the
$(\kappa, \mu)$-nullity distribution, then we have
\begin{equation}\label{EQ11}
R(X, Y )\xi = \kappa \lbrace \eta(Y )X- \eta(X)Y \rbrace +\mu \lbrace \eta(Y )hX- \eta(X)hY \rbrace.
\end{equation}
A contact metric manifold satisfying \eqref{EQ11} is called a $(\kappa, \mu)$-contact metric manifold. If $M$ be a $(\kappa,\mu)$-contact metric manifold, then the following relations hold \cite{11}:
\begin{equation}\label{EQ12}
S(X, \xi) = 2nk\eta(X),
\end{equation}
\begin{equation}
Q\xi= 2nk\xi,
\end{equation}
\begin{equation}\label{EQ13}
h^2 = (k -1)\varphi^{2}, 
\end{equation}
\begin{equation}\label{EQ14}
R(\xi, X)Y =\kappa \lbrace g(X, Y )\xi- \eta(Y )X \rbrace +\mu \lbrace g(hX, Y )\xi- \eta(Y )hX\rbrace ,
\end{equation}
\begin{equation}\label{EQ15}
\begin{array}{ll}
S(X, Y) &= [2(n-1)-n\mu]g(X, Y)+ [2(n-1)+\mu] g(hX,Y)\\
&+[2(1-n)+n(2\kappa+\mu )]\eta(X)\eta(Y).
\end{array}
\end{equation}
Using the equation \eqref{EQ13} and \eqref{EQ15} one can easily get 
\begin{equation}\label{EQ151}
\begin{array}{ll}
S(hX, Y) &= [2(n-1)-n\mu]g(hX, Y)-(\kappa -1) [2(n-1)+\mu] g(X,Y)\\
&+(\kappa -1)[2(n-1)+\mu ]\eta(X)\eta(Y). 
\end{array}
\end{equation} 
We note that if $M^{2n+1}$ be a $(\kappa, \mu)$-contact metric manifold, then $\kappa \leq 1$ \cite{11}. When $\kappa< 1$, the nonzero eigenvalues of $h$ are $\pm\sqrt{1-\kappa}$ each with multiplicity $n$. Let $\lambda$ be the positive eigenvalue of $h$. Then $M^{2n+1}$ admits three mutually orthogonal and integrable distributions $D(0), D(\lambda) $ and $D(-\lambda)$ defined by the eigenspaces of $h$ \cite{18}. From $\lambda=\sqrt{1-\kappa}$ we have
\begin{equation}\label{EQ141}
\kappa=1-\lambda^2.
\end{equation} 
We easily check that Sasakian manifolds are contact $(\kappa, \mu)$-manifolds with $\kappa= 1$ and $h = 0$ \cite{11}. In particular, if  $\mu=0$, then we obtain the condition of $k$-nullity distribution introduced by Tanno\cite{18}. If $\xi \in N(\kappa)$, then we call a contact metric manifold an $N(\kappa)$-contact metric manifold. It is shown that
\begin{theorem}\label{4.0}\cite{30}
Let $M$ be a non-flat $(2n + 1)$-dimensional $N(\kappa)$-contact metric manifold. Then $M$ satisfies the condition 
$R(\xi, X) . R = L(\xi \wedge_S X) . R$ if and only if it is a Sasakian manifold of constant curvature $+1$ or locally isometric to
$E^{n+1} \times S^n(4)$ for $n > 1$.
\end{theorem}


\section{Ricci-generalized pseudosymmetric  $(\kappa,\mu)$-manifolds } 
Blair et al. studied the condition of $(\kappa, \mu)$-nullity distribution on a contact manifold and obtained the following Theorem.
\begin{theorem}\label{5.0} \cite{11}
Let $M^{2n+1}(\varphi,\xi ,\eta, g)$ be a contact manifold with $\xi$ belonging to the $(\kappa, \mu)$-nullity distribution. If $\kappa < 1$, then for any $X$ orthogonal to $\xi$ the following formulas hold:
\begin{itemize}
\item[1.]	The $\xi$-sectional curvature $ K(X,\xi)$ is given by
\begin{equation*}
K(X,\xi)=\kappa +\mu g(hX,X)=\left\lbrace \begin{array}{rl}
&\kappa+ \lambda \mu\quad  \text{if } X\in D(\lambda)\\
&\kappa- \lambda \mu \quad  \text{if } X\in D(-\lambda),
\end{array} \right. 
\end{equation*}
\item[2.] The sectional curvature of a plan section $\lbrace X, Y\rbrace$ normal to $\xi$ is given by
\begin{equation}\label{EQ18}
K(X,Y)=\left\lbrace \begin{array}{rl}
i) & 2(1+ \lambda)- \mu\quad  \text{if } X, Y\in D(\lambda)\\   
ii) & -(\kappa+ \mu)[g(X, \varphi Y)]^2 \quad \text{for any unit vectors } X\in D(\lambda), Y\in D(-\lambda)\\
iii) & 2(1- \lambda)- \mu\quad  \text{if } X, Y\in D(-\lambda),~n> 1.
\end{array} \right.
\end{equation}
\end{itemize}
\end{theorem} 
Let $(\kappa,\mu)$-manifold $M$ satisfies the condition $R .R =L Q(S, R)$. We note that if $L=0$, then $M$ is semisymmetric and then $M$ is either a Sasakian manifold of constant sectional curvature 1, or
locally isometric to the product of a flat $(n+1)$-dimensional Euclidean manifold and an $n$-dimentional manifold of constant curvature 4 \cite{12}. 
 
We suppose $L\neq 0$ and we give some propositions that we need in the sequel.
\begin{proposition}\label{5.2}
A $(\kappa, \mu)$-contact metric manifold $M^{2n+1}$ is Ricci-generalized pseudosymmetric manifold if and only if it satisfies
\begin{equation}\label{EQ100} 
(1-2n)\kappa \mu -n\mu^2 +2(n-1)(\kappa+\mu)=0.
\end{equation}  
\end{proposition}
\begin{proof}
Since the manifold $M$ is Ricci-generalized pseudosymmetric then we have
\begin{equation*}
(R(\xi, X).~R)(Y, Z, W)=L((\xi \wedge_S X).~R)(Y, Z, W).
\end{equation*}
Using \eqref{EQ1}, \eqref{EQ2}, \eqref{EQ8} and taking the inner product with $\xi$, we obtain
\begin{equation}\label{EQ57} 
\begin{array}{ll}
g(R(\xi, X)R(Y,Z)W, \xi)- g(R(R(\xi, X)Y, Z)W, \xi) -g(R(Y, R(\xi, X)Z)W, \xi)\\
-g(R(Y, Z)R(\xi, X)W, \xi) =L\lbrace S(R(Y,Z)W, X)-S(\xi, R(Y,Z)W)\eta(X)\\
-S(X, Y)\eta(R(\xi, Z)W) + S(\xi, Y)\eta(R(X, Z)W) -S(X,Z)\eta(R(Y, \xi)W)\\
+S(\xi, Z)\eta(R(Y, X)W) +S(\xi, W) \eta(R(Y, Z)X)\rbrace.  
\end{array}
\end{equation}
By virtu of \eqref{EQ11}, \eqref{EQ12} and \eqref{EQ14} in the equation \eqref{EQ57}, it follows that
\begin{equation}\label{EQ58} 
\begin{array}{ll}
\kappa g(X, R(Y,Z)W)+\kappa^2(g(X, Z)g(Y, W)-g(X, Y)g(Z, W))+\mu g(hX, R(Y,Z)W)\\
+\kappa \mu\lbrace -g(X, Y)g(hZ, W)- g(hX, Y)g(Z, W)+g(X, Z)g(hY, W)+g(hX, Z)g(Y, W)\\
+\eta(W)\eta(Y)g(hX, Z)-\eta(W)\eta(Z)g(hX, Y) \rbrace +\mu^2\lbrace -g(hX, Y)g(hZ, W)\\
+ g(hX, Z)g(hY, W)+\eta(W)\eta(Y)g(h^2X, Z)-\eta(W)\eta(Z)g(h^2X, Y) \rbrace =\\
L\lbrace S(R(Y,Z)W, X)+ \kappa(-S(X, Y)g(Z, W)+S(X, Y)\eta(W)\eta(Z)+S(X,Z)g(Y, W)\\
-S(X,Z)\eta(W)\eta(Y))+\mu (-S(X, Y)g(hZ, W)+S(X,Z)g(hY, W))\\
+2n\kappa^2(g(X, Z)\eta(W)\eta(Y)-g(X, Y)\eta(W)\eta(Z))+2n\kappa \mu(g(hX, Z)\eta(W)\eta(Y)\\
-g(hX, Y)\eta(W)\eta(Z))\rbrace.  
\end{array} 
\end{equation} 
Taking $Y=\xi$ in equation \eqref{EQ58} and using \eqref{EQ12} and \eqref{EQ14}, one can get
\begin{equation*}
L\eta(W)\lbrace -\kappa S(Z, X)-\mu S(hZ, X)+ 2n\kappa^2 g(X, Z) +2n\kappa \mu g(hX, Z)\rbrace=0.
\end{equation*}
Let $W=\xi$, then since $L\neq 0$ we have
\begin{equation}\label{EQ59}
-\kappa S(Z, X)-\mu S(hZ, X)+ 2n\kappa^2 g(X, Z) +2n\kappa \mu g(hX, Z)=0.
\end{equation}
Putting $Z=X$ and using \eqref{EQ15} and \eqref{EQ151}, equation \eqref{EQ59}  reduces to
\begin{equation}\label{EQ60} 
\begin{array}{ll} 
\kappa [2(n-1)-n\mu]g(X, X)+ \kappa [2(n-1)+\mu]g(hX, X)+\kappa[2(1-n)\\
+n(2\kappa+\mu)]\eta(X)\eta(X)=[\mu(\kappa -1)[2(n-1)+\mu]+2n\kappa^2]g(X, X)\\
+ [-\mu[2(n-1)-n\mu]+2n\kappa \mu]g(hX, X)-\mu(\kappa -1)[2(n-1)+\mu]\eta(X)\eta(X).
\end{array}
\end{equation} 
Then we have 
\begin{equation}\label{EQ61}
\kappa[2(1-n)+n(2\kappa+\mu)]=-\mu(\kappa -1)[2(n-1)+\mu],
\end{equation} 
\begin{equation}\label{EQ62}
\kappa [2(n-1)+\mu]=-\mu[2(n-1)-n\mu]+2n\kappa \mu,
\end{equation}
\begin{equation}\label{EQ63}
\kappa [2(n-1)-n\mu]=\mu(\kappa -1)[2(n-1)+\mu]+2n\kappa^2.
\end{equation}
Subtracting two last equations and replacing the result equation in \eqref{EQ61} we get the result. The converse statement is trivial and this complete the proof.
\end{proof}
From the above proposition and Theorem \ref{4.0}, we have the following corollary.
\begin{corollary}\label{1.1}
Every  three-dimentional $(\kappa,\mu)$-manifold is Ricci-generalized pseudosymmetric manifold if and only if either $M$ is a Sasakian manifold of constant curvature $+1$ or $\kappa=-\mu$ holds on $M$. 
\end{corollary}

\begin{proposition}\label{5.1}
Let $M^{2n+1}$ be a $(\kappa,\mu)$-contact metric Ricci-generalized pseudosymmetric  manifold. Then for any unit vector fields $X, Y\in\mathfrak{X}(M)$ orthogonal to $\xi$ and such that $g(X ,Y)=0$ we have:
\begin{equation}\label{EQ19} 
\begin{array}{ll}
\kappa g(X, R_{XY}Y)+\mu g(hX, R_{XY}Y)- [\kappa+\mu g(hX,X)][\kappa+\mu g(hY, Y)]+\mu^2 g^2(hX,Y)\\
=L\lbrace S(X, R_{XY}Y)-\kappa S(X, X)- \mu S(X, X)g(hY, Y)+ \mu S(X,Y)g(hX,Y)\rbrace .
\end{array}
\end{equation}  
\end{proposition}
\begin{proof}
Since $M$ is Ricci-generalized pseudosymmetric manifold then  
\begin{equation}\label{EQ20}
(R(\xi, X). R)(X, Y)Y=L[((\xi \wedge_S X). R)(X, Y)Y]. 
\end{equation} 
By virtu of \eqref{EQ1}, \eqref{EQ2} and \eqref{EQ8} in \eqref{EQ20} one can easily get
\begin{equation}\label{EQ21} 
\begin{array}{ll}
R(\xi, X). R(X, Y)Y- R(R_{\xi X}X,Y)Y- R(X, R_{\xi X}Y)Y- R(X, Y)R_{\xi X}Y\\ 
=L\lbrace S(X, R(X,Y)Y)\xi - 2n\kappa \eta(R(X,Y)Y)X- S(X,X)R(\xi, Y)Y-\\
S(X,Y)R(X, \xi)Y- S(X,Y)R(X, Y)\xi\rbrace.
\end{array}
\end{equation}
Using \eqref{EQ12} and \eqref{EQ14} in \eqref{EQ21} and taking the inner product with $\xi$ we get the result.
\end{proof}   
\begin{proposition}\label{5.11}
Every Ricci-generalized pseudosymmetric Sasakian manifold with $L\neq \dfrac{1}{2(n-1)-n\mu}$ is of constant curvature 1.
\end{proposition}
\begin{proof}
Let $X$ and $Y$ be unit tangent vectors such that $\eta(X)=\eta(Y)=0$ and $g(X, Y)=0$. Since $M$ is Sasakian then $\kappa =1$ and $h=0$. Using \eqref{EQ10} and \eqref{EQ14} in equation \eqref{EQ21} and direct computation we get  
\begin{equation*}
\begin{array}{ll}
g(X, R(X,Y)Y)\xi- \eta(R(X,Y)Y) X-\xi=L\lbrace S((X, R(X,Y)Y)\xi- \\
2n\eta(R(X,Y)Y)X- S(X, X)\xi \rbrace. 
\end{array}
\end{equation*}  
Taking the inner product with $\xi$ and using \eqref{EQ15} gives $K(X,Y)=1$ and this complete the proof.
\end{proof}
Now we prove the main Theorem of this paper.
\begin{theorem}\label{5.6}
Let $M^{2n+1}$, $n > 1$ be a non-Sasakian Ricci-generalized pseudosymmetric $(\kappa,\mu)$-contact metric manifold. Then either
\begin{itemize} 
\item[1)]  $\kappa =0,~\mu=\dfrac{2n-2}{n},~L=\dfrac{1}{n+1}$ ~or
\item[2)] $\kappa =\dfrac{-2}{n},~\mu=2,~L=\dfrac{1}{n}$ 
\end{itemize} 
hold on $M$.
\end{theorem}   
\begin{proof}
Applying the relation \eqref{EQ19} for $hX=\lambda X,~hY=\lambda Y$ and using \eqref{EQ15} we get
\begin{equation}\label{EQ25}
\begin{array}{ll} 
[\kappa+\lambda \mu- L(2(n-1)-n\mu+\lambda[2(n-1)+\mu])(K(X,Y)-(\kappa+\lambda \mu))=0.
\end{array}
\end{equation} 
Then 
\begin{equation}\label{EQ26}
i)~~K(X, Y)=\kappa + \lambda \mu ~~ \mbox{ or } ~~ ii)~~\kappa= -\lambda \mu+L(2(n-1)-n\mu+\lambda[2(n-1)+\mu]).
\end{equation}
Comparing parts $(i)$ of equations \eqref{EQ18} and \eqref{EQ26} gives
\begin{equation}\label{EQ27}
\mu= 1+\lambda.
\end{equation}
Suppose now that $X, Y \in D(-\lambda)$. Then from equation \eqref{EQ19} and \eqref{EQ15} we have
\begin{equation}\label{EQ28}
[\kappa-\lambda \mu- L(2(n-1)-n\mu-\lambda[2(n-1)+\mu])(K(X,Y)-(\kappa-\lambda \mu))=0,
\end{equation}  
and then
\begin{equation}\label{EQ29}
i)~~K(X, Y)=\kappa - \lambda \mu ~~ \mbox{ or } ~~ ii)~~\kappa= \lambda \mu+L(2(n-1)-n\mu-\lambda[2(n-1)+\mu]).
\end{equation}
Comparing the equations \eqref{EQ18}(iii) and \eqref{EQ29}(i) we have  
\begin{equation}\label{EQ31} 
i)~~\mu= 1-\lambda ~~ \mbox{ or } ~~ ii)~~\lambda=1.
\end{equation}
In the case $X \in D(\lambda)$ and $ Y\in D(-\lambda)$  we get
\begin{equation}\label{EQ34}  
i)~K(X, Y)= \kappa- \lambda \mu ~~ \mbox{ or } ~~  \kappa= -\lambda \mu+L(2(n-1)-n\mu+\lambda[2(n-1)+\mu]),
\end{equation}
while if $X \in D(-\lambda)$ and $Y\in D(\lambda)$ we similarly prove that
\begin{equation}\label{EQ35} 
i)~K(X, Y)= \kappa+ \lambda \mu ~~ \mbox{ or } ~~  \kappa=\lambda \mu+L(2(n-1)-n\mu-\lambda[2(n-1)+\mu]).
\end{equation} 
Let 
$$-\lambda \mu+L(2(n-1)-n\mu+\lambda[2(n-1)+\mu])=A,$$ 
$$\lambda \mu+L(2(n-1)-n\mu-\lambda[2(n-1)+\mu])=B.$$
By the combination now of the equation \eqref{EQ26}(ii), \eqref{EQ27}, \eqref{EQ29}(ii), \eqref{EQ31}, \eqref{EQ34}  and \eqref{EQ35} we establish the following systems among the unknowns $\kappa, \lambda, \mu$ and $L$.
\begin{itemize}
\item[1)] $\lbrace \kappa=A,~ \kappa=B,~ K(X, Y)=\kappa-\lambda\mu,~ K(X, Y)=\kappa+\lambda\mu \rbrace $ 
\item[2)] $\lbrace \kappa=A,~ \kappa=B,~ K(X, Y)=\kappa-\lambda\mu \rbrace $
\item[3)] $\lbrace \kappa=A,~ \kappa=B, ~ K(X, Y)=\kappa+\lambda\mu \rbrace $
\item[4)] $\lbrace \kappa=A,~ \kappa=B \rbrace $
\item[5)] $\lbrace \mu=1+\lambda,~ \mu=1-\lambda,~K(X, Y)=\kappa-\lambda\mu,~ K(X, Y)=\kappa+\lambda\mu \rbrace $
\item[6)] $\lbrace \mu=1+\lambda,~ \mu=1-\lambda,~\kappa=B,~ K(X, Y)=\kappa-\lambda\mu \rbrace $
\item[7)] $\lbrace \mu=1+\lambda,~ \mu=1-\lambda,~\kappa=A,~ K(X, Y)=\kappa+\lambda\mu \rbrace $
\item[8)] $\lbrace  \mu=1+\lambda,~ \mu=1-\lambda ,~\kappa=A,~ \kappa=B \rbrace $
\item[9)] $\lbrace \kappa=A,~\mu=1-\lambda,~K(X, Y)=\kappa-\lambda\mu,~ K(X, Y)=\kappa+\lambda\mu \rbrace $
\item[10)] $\lbrace \kappa=A,~\mu=1-\lambda,~K(X, Y)=\kappa-\lambda\mu,~ \kappa=B \rbrace $
\item[11)] $\lbrace \kappa=A,~\mu=1-\lambda, ~ K(X, Y)=\kappa+\lambda\mu \rbrace $
\item[12)] $\lbrace \kappa=A,~\mu=1-\lambda,~ \kappa=B \rbrace $ 
\item[13)] $\lbrace \kappa=A,~\lambda=1,~K(X, Y)=\kappa-\lambda\mu,~ K(X, Y)=\kappa+\lambda\mu  \rbrace $ 
\item[14)] $\lbrace \kappa=A,~\lambda=1,~K(X, Y)=\kappa-\lambda\mu,~ \kappa=B \rbrace $
\item[15)] $\lbrace  \kappa=A,~\lambda=1,~K(X, Y)=\kappa+\lambda\mu \rbrace $
\item[16)] $\lbrace \kappa=A,~\lambda=1,~ \kappa=B \rbrace $
\item[17)] $\lbrace \mu=1+\lambda,~\lambda=1,~K(X, Y)=\kappa-\lambda\mu,~ K(X, Y)=\kappa+\lambda\mu \rbrace $
\item[18)] $\lbrace \mu=1+\lambda,~\lambda=1,~K(X, Y)=\kappa-\lambda\mu,~\kappa=B \rbrace $
\item[19)] $\lbrace \mu=1+\lambda,~\lambda=1,~\kappa=A,~ K(X, Y)=\kappa+\lambda\mu \rbrace $ 
\item[20)] $\lbrace \mu=1+\lambda,~\lambda=1,~\kappa=A,~\kappa=B \rbrace $
\item[21)] $\lbrace \mu=1+\lambda,~\kappa=B,~K(X, Y)=\kappa-\lambda\mu,~ K(X, Y)=\kappa+\lambda\mu\rbrace $
\item[22)] $\lbrace \mu=1+\lambda,~\kappa=B,~K(X, Y)=\kappa-\lambda\mu\rbrace $ 
\item[23)] $\lbrace \mu=1+\lambda,~\kappa=B,~\kappa=A,~K(X, Y)=\kappa+\lambda\mu \rbrace $
\item[24)] $\lbrace \mu=1+\lambda,~\kappa=B,~\kappa=A \rbrace $.
\end{itemize} 
We note that in all systems that two equations $K(X, Y)=\kappa-\lambda\mu$ and $K(X, Y)=\kappa+\lambda\mu$ hold, one can easily get $\lambda\mu=0$ and since $M$ is non-Sasakian $(\kappa,\mu)$-manifold then $\kappa\neq 1(\Rightarrow \lambda=\sqrt{1-\kappa}\neq 0)$ and hence $\mu=0$.

Also in system with two equations $\kappa=A$ and $\kappa=B$, it follows that
\begin{equation}\label{EQ36} 
-\mu+L[2(n-1)+\mu]=0.   
\end{equation}  
In such systems, if $\mu=0$, then $L=0$ and this is a contradiction with our assumption. Hence $\mu\neq 0$ and from equation \eqref{EQ36} we have $L\neq 1$ and 
\begin{equation}\label{EQ37}  
\mu=\dfrac{2(n-1)L}{1-L}. 
\end{equation} 
In first system, according to what discuss in above, we have $\mu=0$ and $\mu\neq 0$ an this is a contradiction. 

In system 2, for any $X \in D(\lambda),~K(X, \varphi X)=\kappa-\lambda\mu$.  Copmparing this equation by part (ii) of  \eqref{EQ18} implies
\begin{equation}\label{EQ38} 
2\kappa-\lambda\mu+\mu=0.
\end{equation}
Equation \eqref{EQ38} by means of $\kappa=A$ and \eqref{EQ37} gives
\begin{equation}\label{EQ39} 
L=\dfrac{3-\lambda}{2(1+n)}. 
\end{equation} 
Moreover, by virtue of \eqref{EQ141} ,\eqref{EQ37} and \eqref{EQ39} in $\kappa=A$ we get
\begin{equation}\label{EQ391} 
(\lambda-1)(\lambda^2+(n+1)\lambda+(5n-4))=0. 
\end{equation} 
But $(\lambda^2+(n+1)\lambda+(5n-4))\neq0$, because otherwise 
 $$\lambda=\dfrac{1}{2}(-(n+1)+\sqrt{n^2-18n+17}),$$ 
and since $\lambda> 0$ we obtain $-20n+16>0$ and this is impossible. Then $\lambda=1$ and hence from \eqref{EQ141}, \eqref{EQ39} and \eqref{EQ37} we get $\kappa=0$, $L=\dfrac{1}{1+n}$ and $\mu=\dfrac{2n-2}{n}$, respectively. 

In system 3, from $K(X, \varphi X)=\kappa+\lambda\mu$ and part (ii) of equation \eqref{EQ18} we have 
\begin{equation}\label{EQ42}  
2\kappa+\mu(1+\lambda)=0.   
\end{equation}
Replacing \eqref{EQ141} in \eqref{EQ42} gives us
\begin{equation}\label{EQ381} 
\mu=2(\lambda -1).
\end{equation}
On the other hand, applying $\kappa=A$ and \eqref{EQ37} in \eqref{EQ42} we get $L=\dfrac{3+\lambda}{2(n+1)}$ and then from \eqref{EQ37}, $\mu=\dfrac{2(n-1)(\lambda +3)}{2n-1-\lambda}$.
Comparing last equation by \eqref{EQ381} yields 
\begin{equation}\label{EQ90} 
\lambda= \dfrac{1}{2}(n+1\pm \sqrt{n^2-18n+17}).   
\end{equation}
Now using \eqref{EQ141} and \eqref{EQ381}  in \eqref{EQ100} one can get $\lambda= \dfrac{n}{2n-1}$,  that it is a contradiction.

From \eqref{EQ37} and $\kappa=A$ in system 4, it follows that  
\begin{equation}\label{EQ72} 
\kappa= \mu(1-L(1+n)).  
\end{equation}
Applying \eqref{EQ72} to \eqref{EQ100} gives
 \begin{equation}\label{EQ73} 
\mu=\dfrac{-2(n-1)(2-L(1+n))}{(1-3n)-L(1+n)(1-2n)}.  
\end{equation}
Comparing \eqref{EQ37} and \eqref{EQ73} yields $(n^2+n)L^2-(2n+1)L+1=0$.
This quadratic equation has two roots $L_1=\dfrac{1}{1+n}$ and $L_2=\dfrac{1}{n}$. If $L=L_1$ (resp $L=L_2$) equations \eqref{EQ37} and \eqref{EQ72} imply $\mu=\dfrac{2n-2}{n}$ and $\kappa=0$ (resp $\mu=2$ and $\kappa=\dfrac{-2}{n}$ ), respectively. 

In systems 5, 6, 7 and 8 we get easily $\lambda=0$ and then from equation \eqref{EQ141}, $\kappa=1$ which is a contradiction since we suppose $M$ be non-Sasakian.

In systems 9 and 13, 
we have $\mu=0,~\lambda=1$ and then $\kappa=0$. Hence from equation $\kappa=A$ we get $L=0$, contradicting our assumption.
 
In system 10 (resp. system 11), taking $\mu=1-\lambda$ in equation \eqref{EQ38} (resp. \eqref{EQ42}) and using \eqref{EQ141} we have $\kappa=0$ and then $\lambda=1$,~$\mu=0$ and $L=0$ which is a contradiction with $L\neq 0$.

In system 12 (resp. system 24), 
taking $\mu=1-\lambda$ (resp. $\mu=1+\lambda$ ) in \eqref{EQ37} implies 
\begin{equation}\label{EQ44} 
(1-\lambda)(1-L)=2(n-1)L ~~~(resp.~ (1+\lambda)(1-L)=2(n-1)L).
\end{equation} 
Applying \eqref{EQ37} and \eqref{EQ141} to $\kappa=A$ yields
\begin{equation}\label{EQ45}  
(1-\lambda^2)(1-L)=2(n-1)L(1-L(1+n)).
\end{equation}
Now using \eqref{EQ44} in \eqref{EQ45} gives us 
\begin{equation}\label{EQ46}
\lambda=-(1+n)L~~~(resp.~ \lambda=(1+n)L).
\end{equation}
By virtue of $\mu=1-\lambda$ (resp. $\mu=1+\lambda$) and \eqref{EQ46}  in \eqref{EQ36} we obtain
\begin{equation}\label{EQ02}
\lambda=\dfrac{1}{2}(n-2+\sqrt{n^2+8})~~~(resp.~ \lambda=\dfrac{1}{2}(2-n+\sqrt{n^2+8})).
\end{equation}

Now replacing $\kappa=1-\lambda^2$ and $\mu=1-\lambda$ (resp. $\mu=1+\lambda$) in  \eqref{EQ100} we have
\begin{equation*}
(\lambda-1)((1-2n)\lambda^2 +(2-3n)\lambda+(3-n))=0 ~(resp.~ (1+\lambda)((2n-1)\lambda^2 +(2-3n)\lambda+(n-3))=0).
\end{equation*}
If $\lambda=1$ then we have $\kappa=0,~\mu=0$ and hence $L=0$, contradicting our assumption ( resp. and since $\lambda> 0$),  then
\begin{equation*}  
(1-2n)\lambda^2 +(2-3n)\lambda+(3-n)=0~~(resp. ~(2n-1)\lambda^2 +(2-3n)\lambda+(n-3)=0),
\end{equation*}
and $\lambda=\dfrac{1}{2-4n}(3n-2\pm \sqrt{n^2+16n-8})$ (resp. $\lambda=\dfrac{1}{4n-2}(3n-2\pm\sqrt{n^2+16n-8})$). But this is a contradiction with \eqref{EQ02}. 

In system 14 we have $\lambda=1$, $\kappa=0$ and for any $X \in D(\lambda)$ and $Y \in D(-\lambda)$, $K(X, Y)=-\mu$. Comparing last equation with part (ii) of equation \eqref{EQ18} we get  for all $X, Y \in D(\lambda),~g(X, Y)=1$ and this is a contradiction. 
 
In system 15, from $\lambda=1$ we have $\kappa=0$ and $K(X, \varphi X)=\mu$. However, $K(X, \varphi X)= -(\kappa+ \mu)[g(X, \varphi Y)]^2=-\mu$, then $\mu=0$ and from $\kappa=A$ we get $L=0$, contradicting our assumption.

In system 16, substituting $\lambda =1$ in $\kappa=A$ and $\kappa=B$ and summing two result equations, gives us
\begin{equation}\label{EQ65} 
\mu (1-L(n+1))=0,
\end{equation} 
From equation \eqref{EQ65}, since $\mu\neq 0$, we get $L=\dfrac{1}{n+1}$. However, replacing $\kappa=0$ in \eqref{EQ100} we obtain $\mu=\dfrac{2n-2}{n}$.

Systems 17, 18, 19 and 20 can not occur, because in these systems $\lambda=1$ yields $\kappa=0$ and $\mu=1+\lambda=2$.
On the other hand, substituting $\mu=2$ in \eqref{EQ100} gives us $\kappa=\dfrac{-2}{n}$ and this is a contradiction.





In system 21, since $\mu=0$ then $\lambda=-1$, which isn't acceptable because $\lambda >0$. 
  
In system 22 (resp 23), replacing $\mu=1+\lambda$ in \eqref{EQ38} (resp \eqref{EQ42}) and using \eqref{EQ141} we get $\kappa=0$ and then $\lambda=1, \mu=2$ (resp $\lambda=3$ and then $\mu=4, \kappa=-8$) which these solutions don't satisfy the equation \eqref{EQ100} and this complete the proof.
\end{proof}
\begin{example}\cite{30} 
Let $M$ be a three dimensional manifold admitting the Lie algebra structure 
\begin{equation}\label{EQ66} 
 [e_2, e_3] = c_1e_1,~~~ [e_3, e_1] = c_2e_2,~~ [e_1, e_2] = c_3e_3,
\end{equation}
and $\eta$ be the dual $1$-form to the vector field $e_1$. From \eqref{EQ66} we get   
\begin{equation*}\label{EQ 4.3}
 \left\{ 
\begin{array}{rl}
&d\eta(e_2, e_3) = -d\eta(e_3, e_2) =\dfrac{c_1}{2}\neq 0 \\ 
&d\eta(e_i, e_j) = 0\quad  \text{for }~~ (i, j)\neq (2, 3), (3, 2).
\end{array} \right.
\end{equation*}
One can easily check that $\eta$ and $e_1$ are contact form and the characteristic vector field respectively. A Riemannian metric $g$
defining by $g(e_i, e_j) = \delta_{ij}$ is an associated metric if we have $\phi^2 = -I +\eta \otimes e_1$. Then for $c_1 = 2$, $(\varphi, e_1, \eta, g)$ will be a contact metric structure.
Recall that the unique Riemannian connection $\nabla$ of $g$ is given by
\begin{equation*}
\begin{array}{ll} 
2g(\nabla_X Y, Z) &= Xg(Y, Z) + Y g(Z, X) - Zg(X, Y )\\
& - g(X, [Y, Z])- g(Y, [X,Z]) + g(Z, [X, Y ]).
\end{array}
\end{equation*}
Direct calculation implies
\begin{equation*}
\begin{array}{ll}
\nabla_{e_1}e_1 = 0,~~~ \nabla_{e_2}e_2 = 0,~~~ \nabla_{e_3}e_3 = 0,\\
\nabla_{e_1}e_2 =\dfrac{1}{2}(c_2 + c_3 - 2)e_3,~~~ \nabla_{e_2}e_1 =\dfrac{1}{2}(c_2 - c_3 - 2)e_3,\\
\nabla_{e_1}e_3 = -\dfrac{1}{2}(c_2 + c_3 - 2)e_2,~~~ \nabla_{e_3}e_1 =\dfrac{1}{2}(2 + c_2 - c_3)e_2,\\
R(e_2, e_1)e_1 = [1 - \dfrac{(c_3 - c_2)^2}{4}]e_2 + [2 - c_2 - c_3]he_2,\\
R(e_3, e_1)e_1 = [1 -\dfrac{(c_3 - c_2)^2}{4}]e_3 + [2 - c_2 - c_3]he_3,\\
R(e_2, e_3)e_1 = 0.
\end{array}
\end{equation*} 
On the other hand from \eqref{EQ0165} we have 
\begin{equation*}
\nabla_{e_2}e_1 = -\varphi e_2 - \varphi he_2.
\end{equation*}
Comparing two relations of $\nabla_{e_2}e_1$ we conclude that 
\begin{equation*}
 he_2 =\dfrac{c_3 - c_2}{2}e_2,~~~~~~~and ~hence~~~he_3 =-\dfrac{c_3 - c_2}{2}e_3.
\end{equation*}

Hence $e_i$ are eigenvectors of $h$ and $(0, \lambda, -\lambda)$ are their corresponding eigenvalues  where $\lambda =\dfrac{c_3-c_2}{2}e_2$. 

Putting $\kappa= 1- \dfrac{(c_3-c_2)^2}{4}$ and $\mu= 2-c_2 -c_3$ we conclude that $e_1$ belongs to the $(\kappa, \mu)$-nullity distribution, for any $c_2, c_3$.
Let $c_2 = c_3 = 1$, we get $\kappa = 1$ and $M$ is a sasakian manifold with constant curvature $+1$. Hence from corollary \ref{1.1}, $M^3(\varphi, e_1, \eta, g)$ is a Ricci-generalized pseudosymmetric manifold. 
\end{example}



\end{document}